\documentclass{amsart}

\usepackage{graphicx, amsmath, amssymb, verbatim, tikz,wrapfig,url}

\newtheorem{theorem}{Theorem}[section]
\newtheorem{lemma}[theorem]{Lemma}

\theoremstyle{definition}

\theoremstyle{remark}

\numberwithin{equation}{section}

\newcommand{\pw}{\mathfrak{p}}

\begin{document}

\title{Families of Cyclic Cubic Fields}

\author{Steve Balady}
\address{Department of Mathematics, University of Maryland, College Park, MD 20742}
\email{sbalady@math.umd.edu}

\date{November 15, 2015}


\begin{abstract}
We describe a procedure for generating families of cyclic cubic fields with explicit fundamental units. This method generates all known families and gives new ones.
\end{abstract}

\maketitle

In \cite{shanks}, Shanks considered what he termed the ``simplest cubic fields,'' defined as the splitting fields of the polynomials
\begin{equation} \label{sn}
S_n = X^3+(n+3)X^2+nX-1.
\end{equation}
In particular, he showed that if the square root of the polynomial discriminant is squarefree, then the roots of $S_n$ form a system of fundamental units for its splitting field. The analysis of this family was extended by Lettl \cite{lettl} and Washington \cite{wash87}. Lecacheux \cite{lecacheux}, and later Washington \cite{wash96}, discovered a second one-parameter family with a similar property: if a certain specified chunk of the polynomial discriminant is squarefree, the roots of the polynomial form a system of fundamental units. Kishi \cite{kishi} found a third such family. In the following, we show that there are many, many more families of cubics with this property. The first three sections generalize the procedure of Washington \cite{wash96} and follow the model of that paper. The last section is dedicated to examples: we exhibit a new one-parameter family and describe a method for generating arbitrarily many more.

\section{The Families} \label{families}

Let $f(n)$ and $g(n)$ be polynomials with integral coefficients, and assume that the following condition holds:
\begin{equation}
\label{miracle}
\lambda = \frac{f^3+g^3+1}{fg} \textrm{\ \ is a polynomial with integral coefficients.}
\end{equation}
Examples will be given in Section \ref{examples}. For now we remark only that this condition implies that $f|(g^3+1)$ and $g|(f^3+1)$; in particular, $f$ and $g$ have no common factors. If Condition \ref{miracle} is satisfied, the pair $(f,g)$ determines a one-parameter family of polynomials as follows:
\begin{gather*}
P_{f,g}(X) = X^3+a(n)X^2 + \lambda(n) X - 1, \textrm{ where} \\
a = 3(f^2+g^2-fg)-\lambda (f+g).
\end{gather*}

Note that $P_{f,g}$ is symmetric in $f$ and $g$, so we'll assume that $\deg{f} \leq \deg{g}$. If this inequality is strict, then $\deg{\lambda} < \deg{a}$. Together with the rational root theorem, this implies that $P_{f,g}$ is irreducible for all but a small finite list of $n \in \mathbb{Z}$. For the rest of this paper, we will make the standing assumptions that $\deg{f} < \deg{g}$ and then fix an integer $n$ for which $P_{f,g}$ is irreducible. This is practical for theoretical purposes, though we note that the case where both $f$ and $g$ are constant is also of potential interest. 

The discriminant of $P_{f,g}$ is 
\begin{equation*}
D_P = (f-g)^2(3a+\lambda^2)^2 \neq 0,
\end{equation*}
so $P_{f,g}$ determines a cyclic cubic field which we denote $K_{f,g}$ (or sometimes just $K$). Thus $P_{f,g}$ has three real roots which we denote $\theta_1, \theta_2, \theta_3$. Since the constant term of $P_{f,g}$ is a unit in $\mathbb{Z}$, these roots are units in the ring of integers $\mathcal{O}_{K_{f,g}}$.
\begin{lemma} \label{galois}
The $\mathbb{Z}_3$ action of the Galois group on the roots of $P_{f,g}$ is given by
\begin{equation*}
G(\theta) = \frac{f\theta-1}{(f^2+g^2-fg)\theta-g}.
\end{equation*}
\end{lemma}
\begin{proof}
Assume $P(\theta) = 0$. A messy but straightforward calculation shows that $\theta \neq G(\theta)$ (here we use that $f\neq g$ and $G(\theta)$ is real) and that $P(G(\theta)) = 0$.
\end{proof}

Condition \ref{miracle} and Lemma \ref{galois} might seem a bit miraculous (or at least deeply unmotivated). How could we have guessed that Condition \ref{miracle} would lead to a cyclic cubic whose roots are units in $\mathcal{O}_K$ if we didn't already know that it did? It's possible in retrospect to intuit this from the work of Kishi \cite{kishi}, but we originally discovered it using the language of elliptic surfaces as follows.

If $X^3+aX^2+\lambda X-1$ generates a cyclic cubic field, its discriminant is a square: that is, there exists $b \in \mathbb{Q}$ such that
\begin{equation} \label{elliptic}
b^2 = 4a^3+\lambda^2a^2-18\lambda a - 4\lambda^3-27.
\end{equation}
If we fix $\lambda \in \mathbb{Q}$, Equation \ref{elliptic} defines an elliptic curve on which $(a,b)$ is a rational point. Treating $\lambda$ as a parameter in $P^1(\mathbb{C})$ gives the equation of an elliptic surface $W$. In homogeneous coordinates, $W$ is
\begin{equation*}
b^2c = 4a^3+\lambda^2a^2c-18\lambda ac^2 - (4\lambda^3+27)c^3.
\end{equation*}
This surface is birationally equivalent to $X(3)$, which is described homogeneously as
\begin{equation*}
x^3+y^3+z^3=\lambda xyz.
\end{equation*}
Explicitly, the map\footnote{The observation that $W$ is a model for $X(3)$ is contained in unpublished notes of Washington.}  from the homogeneous form of $W$ to $X(3)$ is given by 
\begin{align*}
x &= -\lambda a + b -9c, \\
y &= -\lambda a - b - 9c, \\
z &=  6a + 2\lambda^2c.
\end{align*}

The coordinates of Shanks's simplest cubics $S_n$ (described by Equation \ref{sn}) on $W$ are $[a:b:c;\ \lambda] = [n+3:n^2+3n+9:1;\ n]$. This is sent by the above map to the constant section $[x:y:z;\ \lambda]=[0:-1:1;\ n]$ on $X(3)$. We can do the same with the family of Lecacheux \cite{lecacheux} (presented here in the form given by Washington \cite{wash96}):
\begin{equation*}
L_n = X^3-(n^3-2n^2+3n-3)X^2-n^2X-1.
\end{equation*}
The $W$ coordinates of $L_n$ are
\begin{equation*}
[a:b:c;\ \lambda] = [-n^3+2n^2-3n+3:(n-1)(n^2+3)(n^2-3n+3):1;\ -n^2],
\end{equation*}
which are considerably simpler on $X(3)$:
\begin{equation*}
[x:y:z;\ \lambda] = [-1:-n:1;\ -n^2].
\end{equation*}
Kishi's \cite{kishi} family $K_n$ also takes a simple form on $X(3)$: 
\begin{equation*}
[x:y:z;\ \lambda] = [-n:-n^2-n-1:1;\ -n^3-2n^2-3n-3].
\end{equation*}
The key observation here is that on $X(3)$, each of $S_n$, $L_n$, and $K_n$ are of the form $[x:y:z;\ \lambda] = [f(n):g(n):1;\ \lambda]$ for some polynomials $f$ and $g$. Further, on $X(3)$ we can solve explicitly for $\lambda$:
\begin{equation*}
\lambda = \frac{f^3+g^3+1}{fg}.
\end{equation*}
Condition \ref{miracle} is exactly what we need to reverse this process. In this language, when $f,g,$ and $\lambda$ are polynomials with integral coefficients, $[f(t):g(t):1;\ \lambda]$ determines an algebraic curve on $X(3)$. Translating back to Weierstrass form gives a family of integral points $[a:b:1;\ \lambda]$: that is, it gives a family of polynomials $P_{f,g} = X^3 + aX^2 + \lambda X - 1$ with integral coefficients and square discriminants.

For a general $X^3+aX^2+\lambda X-1$ generating a cyclic cubic field, the Galois group is generated by the fractional linear transformation
\begin{equation*}
G=
\begin{bmatrix}
f & -h \\
(f^2+g^2-fg)/h & -g
\end{bmatrix}
\end{equation*}
for some integers $f,g,h$, since these represent all elements $G \in \textrm{PGL}_2(\mathbb{Q})$ for which $G^3 = I$. Conversely, if we fix $f,g,h$ with $fgh \neq 0$, we can reconstruct
\begin{equation*}
P_{f/h,g/h} = X^3+\frac{3(f^2+g^2-fg)-\lambda h(f+g)}{h^2}X^2+\frac{f^3+g^3+h^3}{fgh}X-1;
\end{equation*}
that is,
\begin{equation*}
f^3+g^3+h^3 = \lambda fgh.
\end{equation*}
This shows that the equivalence of $W$ with $X(3)$ is actually a consequence of the structure of the Galois action. This also shows that, unless $h=1$ and Condition \ref{miracle} is satisfied, there's no guarantee that the roots of the associated polynomial will be algebraic integers.

\section{The Discriminant}

In this section we analyze the discriminant of $K_{f,g}$. A prime number $p \neq 3$ contributes a factor of exactly $p^2$ to the discriminant if and only if $p$ ramifies (and therefore ramifies tamely) in $K_{f,g}$. If 3 ramifies, it contributes a factor of exactly $3^4$. The only primes that can ramify are those that divide the polynomial discriminant of $P_{f,g}$: that is, those dividing
\begin{equation*}
\sqrt{D_P} = (f-g)(3a+\lambda^2).
\end{equation*}
It's mostly possible to analyze the contributions of the two factors separately, following Washington \cite{wash96} or Kishi \cite{kishi}. Unfortunately this approach runs into problems in the general setting, especially for primes that divide both $fg$ and $\sqrt{D_P}$. We could call these exceptional and treat them separately for fixed $(f,g)$, but this isn't very satisfying.

\begin{wrapfigure}{r}{0.30\textwidth}
\begin{center}
\begin{tikzpicture}[node distance = 2cm, auto]
    \node (Q) {$\mathbb{Q}$};
    \node (w) [above of=Q] {$\mathbb{Q}(\omega)$};
    \node (K) [above of=Q, right of=Q] {$K_{f,g}$};
    \node (wa) [above of=K] {$\mathbb{Q}(\omega,\sqrt[3]{\alpha})$};
    \draw[-] (Q) to node {} (w);
    \draw[-] (Q) to node {} (K);
    \draw[-] (K) to node {} (wa);
    \draw[-] (w) to node {} (wa);
\end{tikzpicture}
\end{center}
\end{wrapfigure}

Instead of the direct approach we do the following. Let $\omega$ be a primitive cube root of unity. Then $P_{f,g}$ determines a Kummer extension of $\mathbb{Q}(\omega)$: that is, an extension of the form $\mathbb{Q}(\omega,\sqrt[3]{\alpha})$ for some $\alpha \in \mathbb{Q}(\omega)$. From there we will descend to $K_{f,g}$.

A Kummer generator $\alpha$ is constructed as follows: let $\theta$ be a root of $P_{f,g}$, so the other two roots are $G(\theta)$ and $G^2(\theta)$. Then 
\begin{equation*}
\alpha = (\theta + \omega G(\theta) + \omega^2 G^2(\theta))^3.\phantom{shiftingthistotheleftabit}
\end{equation*}
This gives a massive rational function in $\theta, f,$ and $g$ that we can reduce using the fact that $\theta$ satisfies $P_{f,g}(\theta) = 0$. When we reduce the numerator and denominator to degree 2 functions of $\theta$, the vast majority of the terms cancel and leave an expression in $f$ and $g$. We have that
\begin{gather*}
3a+\lambda^2 = \beta\overline{\beta}, \textrm{ where } \beta = \lambda - 3f - 3\omega (f-g),
\end{gather*}
and a long and painful calculation (made possible by PARI/GP \cite{PARI2}) shows that
\begin{equation*}
\alpha = (f+\omega g)^3(3a+\lambda^2)\beta = (f+\omega g)^3\beta^2\overline{\beta}.
\end{equation*}

If $\pw \nmid 3$ is any prime in $\mathbb{Q}(\omega)$, $\pw$ ramifies in $\mathbb{Q}(\omega,\sqrt[3]{\alpha})$ if and only if $v_\pw(\alpha)$ is not a multiple of 3: that is, if and only if $\pw$ divides $\alpha$ to a non-cube power. Since $(f+\omega g)^3$ is a perfect cube, we need to consider only those $\pw$ dividing $\beta$. Together with the knowledge of the decomposition of rational primes in $\mathbb{Q}(\omega)$, this allows us to compute $D(K_{f,g})$. We begin with the following special case.

\begin{theorem} \label{regthm}
Let $K_{f,g}$ be as in Section \ref{families}. If $3a+\lambda^2$ is squarefree, then
\begin{equation*}
D(K_{f,g}) = (3a+\lambda^2)^2.
\end{equation*}
\end{theorem}
\begin{proof}
The prime $3$ is special, so we deal with it first. If $3|(3a+\lambda^2),$ then $3|\lambda.$ The definition of $a$ shows that $v_3(a) \geq 1$, so $v_3(3a+\lambda^2) \geq 2$ and $3a+\lambda^2$ is not squarefree. If $3\nmid 3a+\lambda^2$ and 3 ramifies in $K_{f,g},$ then $f \equiv g \mod{3}$ since $D(K_{f,g})$ divides $D_P$. If $f \equiv 0$, then $(f,g)$ does not satisfy Condition \ref{miracle}. If $f \equiv 1$, then $3|(3a+\lambda^2)$ which we already assumed it did not. If $f \equiv 2$, then
\begin{equation*}
P_{f,g} \equiv X^3 + X^2 - X - 1 \equiv (X-1)(X+1)^2 \mod{3},
\end{equation*}
so the roots are not all congruent mod the prime above 3 and therefore cannot possibly be fixed by the Galois action. So 3 is unramified and contributes nothing to $D(K_{f,g})$.

Since $|\beta|^2 = 3a+\lambda^2$, a rational prime $p \neq 3$ ramifies in $K_{f,g}$ only if $p|(3a+\lambda^2)$.
In this case, there is a $\pw$ above $p$ for which $\pw|\beta$. Since $\pw$ is unramified in $\mathbb{Q}(\omega)$ and $3a+\lambda^2$ is squarefree, 
$v_\pw(\beta) = 1, v_\pw(\overline{\beta}) = 0,$ and $v_\pw(\alpha) = 2$. Therefore $p$ ramifies in $K_{f,g}$. 
\end{proof}

Because $\alpha$ is so explicit, we can do much better. The prime $p=3$ is badly behaved, but it can be dealt with directly (as in \cite{wash96} or \cite{kishi}) when we're actually given a fixed family $(f,g)$. If a prime $p\neq 3$ ramifies in $K_{f,g}$, it divides $3a+\lambda^2$. If $p \equiv 2 \mod{3}$, $p$ is inert in $\mathbb{Q}(\omega)$. Let $p^m||\beta$; then $p^m||\overline{\beta}$, so $p^{3m}$ is a removable cube in $\alpha$. 
If $p \equiv 1 \mod{3}$, $p$ splits as $\pw\overline{\pw}$ in $\mathbb{Q}(\omega)$; we choose $\pw$ so that $\pw|\beta$. From here we have two cases: either $\overline{\pw}|\beta$ or it does not. If $\overline{\pw}\nmid\beta$ and $m$ is the exact power of $p$ dividing $3a+\lambda^2$, then above $p^m$ we must have $\pw^{2m}\overline{\pw}^m$ in $\alpha$, which is a removable cube if and only if $3|m$. This leaves us to deal with the case where $p \equiv 1 \mod{3}$ and $\overline{\pw}|\beta$. But in this case the rational prime $p=\pw\overline{\pw}$ divides $\beta=\lambda-3f-3\omega(f-g)$. Since $\{1,\omega\}$ is a $\mathbb{Z}$-basis for $\mathbb{Z}[\omega]$, we must have $p|(f-g)$. Since $f$ and $g$ have no common factors, $p\nmid f$. Substituting $f \equiv g$ in $3a+\lambda^2$ implies
\begin{equation*}
\frac{(f^3-1)^2}{f^4} \equiv 0 \mod{p},
\end{equation*}
so we conclude that $p|(f^3-1)$ and $p|(g^3-1)$, so $p|\gcd{(f^3-1,g^3-1)}$. If the polynomial resultant of $f^3-1$ and $g^3-1$ is nonzero (which seems always to be the case in examples), this last condition is satisfied only for the finite collection of $p|\textrm{res}(f^3-1,g^3-1)$. Therefore, for a family given by polynomials $f,g,$ there are only finitely many rational primes $p$ with $p|\beta$, and these must be treated individually. The remaining primes ramify if and only if they divide $3a+\lambda^2$ to a non-cube power. An example using this discussion to compute $D(K_{f,g})$ exactly for a specific family will be given in Section \ref{examples}.

\section{The Regulator and Fundamental Units}
Our standing assumption is that $\deg{f} < \deg{g}$. This implies that 
\begin{align*}
\deg{\lambda} &= 2\deg{g}-\deg{f} \textrm{ and} \\
\deg{a} &= 3\deg{g} - \deg{f}, \textrm{ so} \\
\deg{(3a+\lambda^2)} &= 4\deg{g} - 2\deg{f}.
\end{align*}
Since $\deg{a} > \deg{\lambda}$, we have
\begin{align*}
P_{f,g}(-a-1) &= -a^2+o(a^2) \textrm{ and}\\
P_{f,g}(-a+1) &= a^2+o(a^2).
\end{align*}
If $n$ is sufficiently large we can choose $\theta_1$ to be a root of $P_{f,g}$ for which $-a-1 < \theta_1 < -a+1$. Since we know the Galois action, we get bounds on all the roots:
\begin{align*}
\log{|\theta_1|} &\approx \log{a(n)} \approx \deg{a}\log{n} = (3\deg{g}-\deg{f})\log{n}, \\
\log{|\theta_2|} &= \log{\frac{f\theta_1-1}{(f^2+g^2-fg)\theta_1-g}} \approx (\deg{f}-2\deg{g})\log{n}, \\
\log{|\theta_3|} &= -\log{|\theta_1||\theta_2|} \approx -\deg{g}\log{n}.
\end{align*}
(These approximations are accurate to $o(\log{n})$.) The polynomial regulator $R_P$ is
\begin{equation*}
R_P = 
\left|\begin{vmatrix}
\log{|\theta_1|} & \log{|\theta_2|} \\ 
\log{|\theta_2|} & \log{|\theta_3|}
\end{vmatrix}\right| \approx
(7\deg^2{g} - 5\deg{g}\deg{f} + \log^2{f})\log^2{n}.
\end{equation*}

Using a result of Cusick \cite{cusick}, we can bound $R_K$, the regulator of $\mathcal{O}_{K_{f,g}}$, in terms of the discriminant $D(K_{f,g})$. If $3a+\lambda^2$ is squarefree, we can apply Theorem \ref{regthm}:
\begin{equation} \label{regbound}
R_K \geq \frac{1}{16}\log^2{(D(K_{f,g})/4)} \approx
(4\deg^2{g} - 4\deg{g}\deg{f} + \deg^2{f})\log^2{n}.
\end{equation}
Take $E_K$ to be the group of units of $K_{f,g}$ and $E_P$ to be the subgroup of $E_K$ generated by $\{\pm 1, \theta_1, \theta_2\}$ (and $\theta_3 = (\theta_1\theta_2)^{-1})$. For $\epsilon > 0$ and sufficiently large $n$,
\begin{equation*}
[E_K:E_P] = \frac{R_P}{R_K} < \frac{7\deg^2{g} - 5\deg{g}\deg{f} + \deg^2{f}}
{4\deg^2{g} - 4\deg{g}\deg{f} + \deg^2{f}} + \epsilon.
\end{equation*}
Set $\deg{f} = \rho\deg{g}$ for $0 \leq \rho < 1$; then for sufficiently small $\epsilon$,
\begin{equation*}
[E_K:E_P] < \frac{7-5\rho+\rho^2}{4-4\rho+\rho^2} + \epsilon < 3,
\end{equation*}
so $[E_K:E_P] = 1$ or $2$. But the index must be a norm from $\mathbb{Z}[\omega]$ by \cite[p. 412]{wash96}, so it can't be 2. We have proved:

\begin{theorem} \label{bigthm}
Let $K_{f,g}$ be as in Section \ref{families}. Assume that $\deg{f} < \deg{g}$, that $3a+\lambda^2$ is squarefree, and that $n$ is sufficiently large. Then $\{\theta_1, \theta_2\}$ forms a system of fundamental units for $K_{f,g}.$
\end{theorem}

\section{Examples} \label{examples}

The machinery of the previous sections applies to all published families of cubic polynomials whose roots form systems of fundamental units. The data of these families are summarized in the following table.

\small
\begin{center}
\begin{tabular}{r|r|r|r|r|r}
 & $f$ & $g$ & $\lambda$ & $3a+\lambda^2$ & Ref. \\
\hline
$S_n$ & 0 & $-1$ & $n$ & $n^2 + 3n + 9$ & \cite{shanks},\cite{wash87} \\
$L_n$ & $-1$ & $-n$ & $-n^2$ & $(n^2+3)(n^2-3n+3)$ & \cite{lecacheux},\cite{wash96} \\
$K_n$ & $-n$ & $-n^2-n-1$ & $-n^3-2n^2-3n-3$ & $(n^2+3)(n^4+n^3+4n^2+3)$ & \cite{kishi} \\
$K'_n$ & $-n$ & $n^3-1$ & $-n^5+2n^2$ & $n^{10}+o(n^{10})$ & \cite{kishi}

\end{tabular}
\end{center}
\normalsize
Shanks's simplest cubic fields $S_n$ are degenerate in multiple ways. Condition \ref{miracle} gives $0/0$, but if we define $\lambda$ to be any polynomial, the machinery works. We choose $\lambda = n$ since any other choice is a subparametrization of this one. The previous section's bounds for the regulator of $K(S_n)$ also give $0/0$, but (excellent) bounds are given in \cite{shanks}.

The above table suggests that there might be a family with $f=-n^2$, and in fact there is. Take $(f,g) = (-n^2,n^3-1)$. Then $\lambda = n^3-4n$ and $P_{f,g}$ is
\begin{equation*}
B_n = X^3 + (n^7 + 2n^6 + 3n^5 - n^4 - 3n^3 - 3n^2 + 3n + 3)X^2 + (-n^4 + 3n)X - 1.
\end{equation*}
Since $\deg{f}$ and $\deg{g}$ are coprime, $B_n$ is not simply a subparametrization of $S_n, L_n,$ or $K_n$. The discriminant is the square of
\begin{equation*}
\sqrt{D_P} = (n^3+n^2-1)(n^4-3n+3)(n^4+3n^3+6n^2+6n+3).
\end{equation*}
The first factor is $g-f$, and the product of the last two is $3a+\lambda^2$. By Theorem \ref{regthm}, whenever $3a+\lambda^2$ is squarefree,
\begin{equation*}
D(B_n) = (n^4-3n+3)^2(n^4+3n^3+6n^2+6n+3)^2.
\end{equation*}
For $0<n<10^4$, it turns out that $3a+\lambda^2$ is squarefree $14.8\%$ of the time. This is perfectly fine, but we can do better. Following the discussion after Theorem \ref{regthm}, we have an algorithm that completely determines $D(B_n)$. 
Write $3a+\lambda^2 = bc^3$ with $b$ cubefree and take $p|(3a+\lambda^2)$. If $p = 3k+2$, then $p$ is unramified. For primes of the form $p=3k+1$, we need to take special care only if $p|\textrm{res}(f^3-1,g^3-1) = -5^3$, which never happens. Therefore $p=3k+1$ ramifies iff $p|b$.

We can also deal with $p=3$ (though admittedly in a more \textit{ad hoc} manner). Straightforward calculations show that
\begin{equation*}
3|D_P \iff v_3(3a+\lambda^2)>0 \iff 3|n \iff v_3(3a+\lambda^2)=2 \iff 3|b.
\end{equation*}
Thus 3 ramifies only if $3|b$. Conversely, writing out $B_{3k}(X+1)$ shows that it satisfies Eisenstein's criterion, so 3 ramifies if $3|b$. We have shown the following.
\begin{theorem}
Write $(n^4-3n+3)(n^4+3n^3+6n^2+6n+3) = bc^3$ with $b$ cubefree. Then
\begin{equation*}
D(B_n) = 81^\delta \prod_{p|b, p = 3k+1}p^2,
\end{equation*}
where $\delta = 1$ if $3|b$ and $\delta = 0$ otherwise.
\end{theorem}

To bound the regulator, we have roots of order approximately $n^7$, $n^{-4}$, and $n^{-3}$. Numerical calculation shows that we get an index of less than 3 whenever $|n| > 4$, and computing the regulators of $K(B_n)$ for $|n|\leq 4$ individually by computer shows that we need only restrict to $n \neq -1$. (The roots of $B_{-1} = X^3 - 3X^2 - 4X - 1$ generate an index 3 subgroup of the full unit group. Since $3a+\lambda^2 = 7$ is squarefree, the restriction is in fact necessary.) We have shown the following.
\begin{theorem}
Let $n \in \mathbb{Z}$ with $n\neq -1$, and suppose $3a+\lambda^2$ is squarefree. Then $\{\pm1,\theta_1,\theta_2\}$ generate the unit group of the ring of integers of $K(B_n)$.
\end{theorem}

For this theorem to be useful, we'd like to know that it gives units other than those produced by $S_n$, $L_n$, or $K_n$. Consider $B_2 = X^3+309X^2-10X-1$. The regulator of $K(B_2)$ is approximately $24.733$. Equation \ref{regbound} tells us that we only need to check the regulators of  $K(L_n)$ and $K(K_n)$ for $|n|<10$. Shanks \cite{shanks} computes that the regulator of $K(S_n)$ is $\log^2{|n|} + o(1)$, which is greater than $K(B_2)$ for $|n|>150$. Checking the regulators for small $n$ in PARI/GP shows that the regulator of the units produced by $B_2$ is not a regulator produced by $S_n, L_n,$ or $K_n$. Therefore, these units are new. (We are not claiming that the field $K(B_2)$ is necessarily distinct from $K(S_n), K(L_n),$ and $K(K_n)$ for all $n$, just that the units are new. It is possible that $K(B_2)$ is somehow isomorphic to one of these fields, but that the units produced there are not fundamental.)

At this point one might suspect that these families are abundant and go hunting for $(f,g)$ with $f=-n^3$. As we've said before, Condition \ref{miracle} implies that $g|(f^3+1)$. This limits the search considerably: a given $f$ determines a finite list of possible $g$. For $f=-n^3$, the only potential candidates for $g$ are $\pm(n-1), \pm(n^2+n+1),$ and $\pm(n^6+n^3+1)$. Of these, the only pair satisfying Condition \ref{miracle} is $(-n^3,-n^6-n^3-1) = K_{n^3}$. Actually, an argument using the factorization of $n^{3k} \pm 1$ shows that this is always the case: if $k>2$, any family of the form $(\pm n^k, g)$ is a subparametrization of the families $K_n$ or $K'_n = (-n,n^3-1)$.

This raises the question of whether all families might be subparametrizations of some finite list. This isn't right either:

\begin{theorem} \label{fgtrick}
Assume $(f,g)$ satisfy Condition \ref{miracle} and $f\neq0$. Then $(g,k)$, where $k=(g^3+1)/f$, also satisfy Condition \ref{miracle}. If $\deg{f}$ and $\deg{g}$ are coprime, then $\deg{g}$ and $\deg{k}$ are coprime.
\end{theorem}
\begin{proof}
A calculation shows that
\begin{equation*}
\frac{g^3+k^3+1}{gk} = \frac{f^3+(g^3+1)^2}{f^2g}.
\end{equation*}
Since $f|(g^3+1), g|(f^3+1)$, and $f$ is coprime to $g$, we have $f^2g|(f^3+(g^3+1)^2)$. The second part of the theorem follows from the fact that $\deg{k} = 3\deg{g}-\deg{f}$.
\end{proof}

We can apply Theorem \ref{fgtrick} to $L_n$, $K_n$, or $B_n$ repeatedly to produce any number of new families. Since the construction is asymmetric in $f$ and $g$, we can also run it backwards by applying it to $(g,f)$. For example, applying the construction forwards iteratively to $B_n$ (more precisely, to $(f,g)=(-n^2,n^3-1)$) produces
\begin{align*}
(g,k_1) &= (n^3-1,-n^7+3n^4-3n), \\
(k_1,k_2) &= (-n^7+3n^4-3n,-n^{18}+o(n^{18})), \\
(k_2,k_3) &=  (-n^{18}+o(n^{18}), n^{47}+o(n^{47}), ...
\end{align*}
(Incidentally the ratio $\deg{g}/\deg{f}$ approaches $\frac{3+\sqrt{5}}{2} = \phi+1 = \phi^2$ as we iterate; this is true in general and a consequence of the degree calculation.) Running $B_n$ backwards gives $B_{-n}$. Running $L_n$ backwards gives $S_n$ (which has $f=0$, so we stop there), while running it forwards gives $K'_n$ and then new families. Running $K_n$ backwards gives the family corresponding to $(-n,n-1)$ and then $K_n$ again, while running it forwards gives new families. The families $K'_n$ and $K_{-n,n-1}$ are discussed briefly in Kishi \cite{kishi}; the rest, as far as I know, are all completely new. $K_{-n,n-1}$ is badly behaved; it has $\lambda = 3$ and always produces units of index 3 in the full unit group. (This does not contradict Theorem \ref{bigthm} since $\deg{f}=\deg{g}$.) Note that we claim only that the families themselves are new. I do not know whether these new families all produce new units, or even whether they all generate distinct fields.

Finally, we might ask if all families of cyclic cubic fields are of the form $P_{f,g}$ for the right choice of $(f,g)$. Here the answer is also negative. For instance, plotting the pairs of $(a,\lambda)$ for which the discriminant of $X^3+aX^2+\lambda X -1$ is a square reveals (among other things) a parabola in the second quadrant. Computation shows that this is the family $X^3+(-n^2+2n-6)X^2 + (n^2+5)X - 1$. The Galois group is generated by the fractional linear transformation
\begin{equation*}
\begin{bmatrix}
n^2-n+1 & -(n^2+n+1) \\
n^2-3n+3 & -2
\end{bmatrix},
\end{equation*}
which cannot possibly come from $f,g \in \mathbb{Z}[n]$. In the language of elliptic surfaces, this is the family $[n^2-n+1 : 2 : n^2+n+1;\ n^2+5]$ on the homogeneous coordinates for $X(3)$. As it happens, we know this family already. Begin with $S_n$, Shanks' simplest cubic fields. If we square the roots (which are units) we get another family $S_n^2$, and that family is the one above.

\section{Acknowledgments}
This paper comprises a portion of the author's forthcoming doctoral dissertation. The author is deeply grateful to his advisor Larry Washington for his advice and encouragement.

\bibliographystyle{plain}
\bibliography{references}
\end{document}